\documentclass[preprint,11pt]{elsarticle}

\usepackage{amsfonts, amsmath, amscd}
\usepackage[psamsfonts]{amssymb}

\usepackage{amssymb}

\usepackage{pb-diagram}

\usepackage[all,cmtip]{xy}

\usepackage[usenames]{color}

\newtheorem{theorem}{Theorem}[section]
\newtheorem{lemma}[theorem]{Lemma}

\newtheorem{definition}[theorem]{Definition}

\newproof{proof}{Proof}

\numberwithin{equation}{section}
\numberwithin{theorem}{section}

\newcommand{\e}{\varepsilon}
\newcommand{\w}{\omega}

\newcommand{\IR}{\mathbb{R}}

\newcommand{\IF}{\mathbb{F}}

\newcommand{\TTT}{\mathcal{T}}

\newcommand{\KK}{\mathcal{K}}

\newcommand{\Ra}{\Rightarrow}

\newcommand{\cacx}{\overline{\mathrm{acx}}}

\newcommand{\CC}{C_k}

\begin{document}

\begin{frontmatter}

\title{A Banach space characterization of (sequentially) Ascoli spaces}

\author{S.~Gabriyelyan}
\ead{saak@math.bgu.ac.il}
\address{Department of Mathematics, Ben-Gurion University of the Negev, Beer-Sheva, P.O. 653, Israel}

\begin{abstract}
We prove that a Tychonoff space $X$ is an Ascoli space (resp., a sequentially Ascoli space) if and only if  for each  Banach space $E$, every $k$-continuous  and almost $k$-compact  (resp., almost $k$-sequential) map $T$ form $X$ into the Banach dual $E'$ of $E$ is continuous.
\end{abstract}


\end{frontmatter}


\section{Introduction}


Let $X$ be a Tychonoff space. A map $S$ from $X$ to a Tychonoff space $Y$ is called {\em $k$-continuous} if each restriction $f{\restriction}_K$ of $f$ to any compact subset $K$ of $X$ is continuous.  We recall that $X$ is called a {\em $k_\IR$-space} if every $k$-continuous function $f:X\to \IR$ is continuous. Denote by $\CC(X)$ the space $C(X)$ of all real or complex valued continuous functions on $X$ endowed with the compact-open topology. It is well known that $X$ is a $k_\IR$-space if and only if $\CC(X)$ is complete.

One of the basic theorems in Analysis is the Ascoli theorem which states that if $X$ is a $k$-space, then every compact subset of $\CC(X)$ is evenly continuous, see Theorem 3.4.20 in \cite{Eng}. Being motivated by the Ascoli theorem we  introduced and studied in \cite{BG} the class of Ascoli spaces.
A Tychonoff space $X$ is called an {\em Ascoli space} if every compact subset $\KK$ of $\CC(X)$  is evenly continuous, that is the map $X\times\KK \ni(x,f)\mapsto f(x)$ is continuous. In other words, $X$ is Ascoli if and only if the compact-open topology of $\CC(X)$ is Ascoli in the sense of \cite[p.45]{mcoy}. One can easily show that $X$ is Ascoli if and only if every compact subset of $\CC(X)$ is equicontinuous. Recall that a subset $H$ of $C(X)$ is {\em equicontinuous} if for every $x\in X$ and each $\e>0$ there is an open neighborhood $U$ of $x$ such that $|f(x')-f(x)|<\e$ for all $x'\in U$ and $f\in H$.  In \cite{Noble} Noble proved that every $k_\IR$-space is Ascoli, however the converse is not true in general, see \cite{BG}.

Being motivated by the classical notion of $c_0$-barrelled locally convex spaces we defined in \cite{Gabr:weak-bar-L(X)}  a Tychonoff space $X$ to be {\em sequentially Ascoli} if every convergent sequence in $\CC(X)$ is equicontinuous. Clearly, every Ascoli space is sequentially Ascoli, but the converse is not true in general (every non-discrete $P$-space is sequentially Ascoli but not Ascoli, see Proposition 2.9 in \cite{Gabr:weak-bar-L(X)}). Ascoli and sequentially Ascoli spaces in various classes of topological, function and locally convex spaces are thoroughly studied in
\cite{BG,Gabr-LCS-Ascoli,Gabr-reflex-L(X),Gabr-seq-Ascoli,Gabr-pseudo-Ascoli}.

A map $T$ from a Tychonoff space $X$ to a locally convex space $E$ is called {\em bounded} if $T(X)$ is a bounded subset of $E$. The space $E$ endowed with the weak topology is denoted by $E_w$. The next characterization of sequentially Ascoli spaces is proved in the recent article \cite{Gabr-lc-Ck}.
\begin{theorem} \label{t:Ck-locally-complete}
A Tychonoff space $X$ is  sequentially Ascoli if and only if every $k$-continuous  bounded map $T:X\to c_0$ is continuous whenever it is continuous as a map from $X$ into  $(c_0)_w$.
\end{theorem}

Observe that $c_0$ is a subspace of the the dual Banach space $\ell_\infty$ of $\ell_1$. Therefore a map $T:X\to c_0$ can be considered as a map from $X$ into the Banach dual $\ell_\infty$ of $\ell_1$. This remark,  the definition of $k_\IR$-spaces and Theorem \ref{t:Ck-locally-complete} motivate
the problem of finding a condition on $k$-continuous maps from a Tychonoff space $X$ into dual Banach spaces which characterizes Ascoli and sequentially Ascoli spaces. We solve this problem in Theorem  \ref{t:Ascoli-Banach}.



\section{Main Result} \label{seq:main}


We start from some necessary definitions and notations. Denote by $\IF$ the field of real or complex numbers.
Let $X$ be a Tychonoff space, and let $f:X\to\IF$ be a function. For a subset $A\subseteq X$ and $\e>0$, let
\[
\|f\|_A:=\sup(\{|f(x)|:x\in A\}\cup\{0\})\in [0,\infty].
\]
Recall that a subset $A\subseteq X$ is called {\em functionally bounded} if $\|f\|_A<\infty$ for any continuous function $f:X\to\IF$. We denote by $C(X)$ the space of continuous $\IF$-valued functions on $X$.  If $\e>0$, we set
\[
[A;\e]:=\{f\in C(X):\|f\|_A\le\e\}.
\]
A family $\mathcal{S}$ of subsets of  $X$ is {\em directed\/} if for any sets $A,B\in\mathcal{S}$, the union $A\cup B$ is contained in some set $C\in\mathcal{S}$. Each directed family $\mathcal{S}$ of functionally bounded sets in $X$ induces a locally convex topology $\TTT_{\mathcal{S}}$ on $C(X)$ whose neighborhood base at zero consists of the sets $[A;\e]$, where $A\in\mathcal{S}$ and $\e>0$. The topology $\TTT_{\mathcal{S}}$ is called {\em the topology of uniform convergence on sets of the family $\mathcal{S}$}. The topology $\TTT_{\mathcal{S}}$ is Hausdorff if and only if the union $\bigcup\mathcal{S}$ is dense in $X$. If $\mathcal{S}$ is the family of all compact subsets of $X$, then the topology $\TTT_{\mathcal{S}}$ will be denoted by $\TTT_k$, and the function space $C_{\TTT_{\mathcal{S}}}(X)$ will be denoted by $\CC(X)$. More generally, if $\TTT$ is a locally convex vector topology on $C(X)$ we put $C_\TTT(X):=\big(C(X),\TTT)$.

We need the following  technical lemma.
\begin{lemma} \label{l:technical}
Let $M>0$ and $0<\delta\leq \tfrac{1}{10} M$. If $z\in \IF$ satisfies $|z|>M-\delta$, then
\[
\Big| \tfrac{z+t}{|z+t|}- \tfrac{z}{|z|}\Big| < \tfrac{2\delta}{M} \;\; \mbox{ for every }\; |t|\leq |\delta|.
\]
\end{lemma}

\begin{proof}
If $\IF=\IR$, then the lemma is trivial because the conditions on $\delta$ and $z$ imply $\tfrac{z+t}{|z+t|}= \tfrac{z}{|z|}$. Assume that $\IF=\mathbb{C}$.
Observe that the distance between the unit numbers $\tfrac{z+t}{|z+t|}$ and $\tfrac{z}{|z|}$ is maximal if and only if the angle $\varphi$ between them takes the maximal value and hence, taking into account that $|t|\leq\delta$, if and only if the ray $\lambda\cdot \tfrac{z+t}{|z+t|}, \lambda>0$, is tangent to the circle of radius $\delta$ centered at $z$. Then $|z|^2=|z+t|^2 +\delta^2$ and $\cos^2\varphi= \tfrac{|z+t|^2}{|z|^2}=1-\tfrac{\delta^2}{|z|^2}$. Therefore
\[
|z|\cdot |1-\cos\varphi|\leq |z|\cdot(1-\cos^2\varphi)=\tfrac{\delta^2}{|z|}=\delta\cdot \tfrac{\delta}{|z|}<\delta\cdot \tfrac{\tfrac{1}{10}M}{M-\tfrac{1}{10}M}=\tfrac{\delta}{9}.
\]
Observe also that $|z+t|\geq |z|-|t|>M-2\delta$. Hence
\[
\begin{aligned}
\Big| \tfrac{z+t}{|z+t|}- \tfrac{z}{|z|}\Big| &\leq \tfrac{|z+t-z\cdot\cos\varphi|}{|z+t|}\leq \tfrac{|z|\cdot |1-\cos\varphi|}{M-2\delta}+\tfrac{|t|}{M-2\delta}\\
& <\tfrac{1}{M-2\delta}\big( \tfrac{\delta}{9}+\delta\big)=\tfrac{10}{9} \cdot \tfrac{\delta}{M-2\delta}<\tfrac{10}{9} \cdot \tfrac{\delta}{M-\tfrac{2}{10}M}<\tfrac{2\delta}{M}.
\end{aligned}\Box
\]
\end{proof}

\begin{lemma} \label{l:restriction-Ck}
Let $X$ be a Tychonoff space, $\mathcal{S}$ be a  directed family of functionally bounded sets in $X$ such that $\bigcup\mathcal{S}$ is dense in $X$, and let $M>0$. Then the co-restriction map  $Q:C_{\TTT_{\mathcal{S}}}(X)\to C_{\TTT_{\mathcal{S}}}(X)$ defined by
\[
Q(f)(x):=\left\{
\begin{aligned}
M\cdot \tfrac{f(x)}{|f(x)|}, & \;\; \mbox{ if } \; |f(x)|> M,\\
f(x), & \;\; \mbox{ if } \; |f(x)|\leq M,
\end{aligned} \right.
\]
is continuous.
\end{lemma}

\begin{proof}
Let $K\in\mathcal{S}$ and $\e>0$. Set $\delta=\tfrac{1}{20}\min\{M,\e\}$. To prove that $Q$ is continuous it suffices to show that
\begin{equation} \label{equ:restriction-Ck-1}
Q\big( f+[K;\delta]\big) \subseteq Q(f)+ [K;\e] \;\; \mbox{ for every }\; f\in C(X).
\end{equation}
Fix $g\in [K;\delta]$ and $x\in K$. Consider the following possible cases.

Assume that $|f(x)+g(x)|\leq M$ and $|f(x)|\leq M$. Then
\[
|Q(f+g)(x)-Q(f)(x)|=|f(x)+g(x)-f(x)|=|g(x)|\leq\delta <\e.
\]

Assume that $|f(x)+g(x)|> M$. Then $|f(x)|\geq |f(x)+g(x)|-|g(x)|>M-\delta.$ Since $\delta < \tfrac{1}{10} M$ we can apply Lemma \ref{l:technical} for $z=f(x)$ and $t=g(x)$.

If $|f(x)+g(x)|> M$ and $|f(x)|> M$, then Lemma \ref{l:technical} implies
\[
|Q(f+g)(x)-Q(f)(x)|=\Big|M\cdot \tfrac{f(x)+g(x)}{|f(x)+g(x)|} -M\cdot \tfrac{f(x)}{|f(x)|} \Big|<M\cdot \tfrac{2\delta}{M}<\e.
\]

If $|f(x)+g(x)|> M$ and $|f(x)|\leq M$, then Lemma \ref{l:technical} implies
\[
\begin{aligned}
|Q(f+g)(x)-Q(f)(x)| &=\Big|M\cdot \tfrac{f(x)+g(x)}{|f(x)+g(x)|}-f(x)\Big|\\
& \leq \Big|M\cdot \tfrac{f(x)+g(x)}{|f(x)+g(x)|} -M\cdot \tfrac{f(x)}{|f(x)|} \Big|+ |f(x)|\cdot \Big| 1 - \tfrac{M}{|f(x)|}\Big|\\
& \leq M\cdot \tfrac{2\delta}{M} + M\cdot \Big| 1 - \tfrac{M}{M-\delta}\Big|=2\delta +\tfrac{2\delta}{M-\delta}<2\delta +M\cdot\tfrac{2\delta}{M-\tfrac{M}{20}}<\e.
\end{aligned}
\]

Assume that $|f(x)+g(x)|\leq M$ and $|f(x)|> M$. Then $|f(x)+g(x)|\geq |f(x)|-|g(x)|>M-\delta$.  Therefore,  for $z=f(x)+g(x)$ and $t=-g(x)$, Lemma \ref{l:technical} implies
\[
\begin{aligned}
|Q(f+g)(x)-Q(f)(x)| &=\Big|f(x)+g(x) - M\cdot \tfrac{f(x)}{|f(x)|}\Big|\\
& \leq \Big|M\cdot \tfrac{f(x)+g(x)}{|f(x)+g(x)|} -M\cdot \tfrac{f(x)}{|f(x)|} \Big|+ |f(x)+g(x)|\cdot \Big| 1 - \tfrac{M}{|f(x)+g(x)|}\Big|\\
& \leq M\cdot \tfrac{2\delta}{M} + M\cdot \Big| 1 - \tfrac{M}{M-\delta}\Big|=2\delta +\tfrac{\delta}{M-\delta}<2\delta +M\cdot\tfrac{\delta}{M-\tfrac{M}{20}}<\e.
\end{aligned}
\]
Consequently, in any possible case we have $|Q(f+g)(x)-Q(f)(x)|<\e$ for every $x\in K$ and each $g\in [K;\delta]$. Thus (\ref{equ:restriction-Ck-1}) is satisfied.\qed
\end{proof}

Let  $X$ be a Tychonoff space, and let $E$ be a locally convex space over the field $\IF$ of real or complex numbers. Denote by $E'$ the topological dual space of $E$ and set $E'_{w^\ast}:=\big(E',\sigma(E',E)\big)$, where $\sigma(E',E)$ is the weak$^\ast$ topology on $E'$. A map  $T:X\to E$  (resp.,  $T:X\to E'$) is called {\em weakly continuous} (resp., {\em weak$^\ast$ continuous}) if $T$ is continuous as a map from $T$ to $E_w$ (resp., to $E'_{w^\ast}$). Any map $T$ from $X$ into the algebraic dual $ E^\ast$ of $E$  defines an {\em associative  map} $T_E:X\times E\to \IF$ by
\[
T_E(x,z):=\langle T(x),z\rangle \quad \mbox{ where }\; x\in X \;\mbox{ and }\; z\in E.
\]
\begin{lemma} \label{l:functional-repr}
A map $T:X\to E'$ is weak$^\ast$ continuous if and only if the associative map $T_E$ is linear by the second argument and separately continuous. In this case the function $T_E(x,z):X\to \IF$ is continuous for every $z\in E$.
\end{lemma}

\begin{proof}
Assume that $T$ is weak$^\ast$ continuous. Then it is evident that $T_E$ is linear and  continuous by the second argument. To prove that $T_E$ is continuous also by the first argument, fix $x_0\in X$, $z_0\in E$ and $\e>0$. As $T$ is weak$^\ast$ continuous, there exists a neighbourhood $V$ of $x_0$ such that $|\langle T(x)-T(x_0),z_0\rangle|<\e$ for every $x\in V$. Then for every $x\in V$, we obtain
\[
|T_E(x,z_0)-T_E(x_0,z_0)|=|\langle T(x)-T(x_0),z_0\rangle|<\e
\]
which means that $T_E(x,z_0)$ is continuous at $x_0$.

Conversely, assume that  $T_E$ is linear by the second argument and separately continuous. The continuity and linearity of $T_E(x,\cdot)$ exactly means that $T(x)\in E'$. To show that $T$ is weak$^\ast$ continuous, fix $x_0\in X$, a finite $F\subseteq E$ and $\e>0$. Since $T_E(\cdot,z)$ is continuous at $x_0$ for every $z\in F$, there is a neighborhood $V$ of $x_0$ such that
\[
|\langle T(x)-T(x_0),z\rangle|=|T_E(x,z)-T_E(x_0,z)|<\e \;\;\mbox{ for every } x\in V \mbox{ and each }  z\in F,
\]
which means that $T(x)-T(x_0)\in [F;\e]$. Thus $T$ is weak$^\ast$ continuous.\qed
\end{proof}

The next notions are well-defined by Lemma \ref{l:functional-repr}.
\begin{definition} \label{def:almost-T-compact} {\em
Let $X$ be a Tychonoff space, $\TTT$ be a locally convex vector topology on $C(X)$, and let $E$ be a locally convex space. We shall say that a continuous map $T:X\to E'_{w^\ast}$ is
\begin{enumerate}
\item[$\bullet$] {\em $\TTT$-compact} if there is a neighborhood $U$ of zero in $E$ such that the family
$
\{ T_E(x,a):\; a\in U\}
$
is relatively compact in $C_\TTT(X)$;
\item[$\bullet$]  {\em almost $\TTT$-compact} (resp., {\em almost $\TTT$-sequential}) if there are a neighborhood $U$ of zero in $E$ and a compact subset (resp., a null sequence) $K$ of $C_\TTT(X)$ such that the family $\{ T_E(x,a):\; a\in U\}$ is contained in the absolutely convex closed hull $\cacx(K)$ of $K$.
\end{enumerate}
If $\TTT=\TTT_k$ we shall say that $T$ is ({\em almost}) {\em $k$-compact} or  ({\em almost})  {\em $k$-sequential}, respectively.\qed}
\end{definition}

If $E$ is a Banach space, we denote by $B_E$ the closed unit ball of $E$. The Banach dual space of $E$ is denoted by $E'_\beta$.
Recall also that for any set $\Gamma$, $\ell_\infty(\Gamma)$ is the Banach dual of the Banach space $\ell_1(\Gamma)$.
Now we are ready to prove the main result of the article.
\begin{theorem} \label{t:Ascoli-Banach}
For a Tychonoff space $X$, the following assertions are equivalent:
\begin{enumerate}
\item[{\rm(i)}] $X$ is an Ascoli {\rm(}resp., sequentially Ascoli{\rm)} space;
\item[{\rm(ii)}] for each  cardinal $\Gamma$, every  $k$-continuous and almost $k$-compact  {\rm(}resp., almost $k$-sequential{\rm)} map $T:X\to \ell_{\infty}(\Gamma)$ is continuous;
\item[{\rm(iii)}]  for each  Banach space $E$, every $k$-continuous  and almost $k$-compact  {\rm(}resp., almost $k$-sequential{\rm)} map $T:X\to E'_{\beta}$ is continuous.
\end{enumerate}
\end{theorem}

\begin{proof}
(i)$\Ra$(ii) Let $T:X\to \ell_{\infty}(\Gamma)$ be a  $k$-continuous and almost $k$-compact  {\rm(}resp., almost $k$-sequential{\rm)} map. We have to show that $T$ is continuous. 
For every $\gamma\in\Gamma$, let $\delta_\gamma=(a_i)\in \ell_{1}(\Gamma)$ be such that $a_i=1$ if $i=\gamma$, and $a_i=0$ otherwise. Set $f_\gamma:=\langle T(x),\delta_\gamma\rangle$. Since $T$ is almost $k$-compact  {\rm(}resp., almost $k$-sequential{\rm)}, we have the following: (a) the map $T$ is weak$^\ast$ continuous and hence all functions $f_\gamma$ are continuous, and (b) there is a compact subset (resp., a null-sequence) $\tilde K$ of $\CC(X)$ such that the family $K_0:=\{f_\gamma: \gamma\in\Gamma\}$ is contained in the absolutely convex closed hull $\cacx(\tilde K)$ of $\tilde K$.  As $X$ is (resp., sequentially) Ascoli, $\tilde K$ and hence also $\cacx(\tilde K)$ and $K_0$ are equicontinuous. Fix an arbitrary $x_0\in X$, and let $\e>0$. Since $K_0$ is equicontinuous, there is a neighborhood $V$ of $x_0$ such that
\[
|f_\gamma(x)-f_\gamma(x_0)|\leq \e \;\; \mbox{ for all }\; x\in V \; \mbox{ and each }\; \gamma\in\Gamma.
\]
Then, for every $x\in V$, we have
\[
\|T(x)-T(x_0)\|_{\ell_{\infty}(\Gamma)}=\sup\{ |f_\gamma(x)-f_\gamma(x_0)|:\; \gamma\in\Gamma\} \leq \e
\]
which means that $T(x)-T(x_0)\in\e B_{\ell_{\infty}(\Gamma)}$. Thus $T$ is continuous.
\smallskip

(ii)$\Ra$(iii) Let $E$ be a Banach space, and let $T:X\to E'$ be a $k$-continuous and almost $k$-compact  {\rm(}resp., almost $k$-sequential{\rm)} map. It is well known that there is a set $\Gamma$ such that $E$ is isometric to a quotient space of $\ell_{1}(\Gamma)$. Let $S:\ell_{1}(\Gamma) \to E$ be an isometric quotient map, so $S(B_{\ell_{1}(\Gamma)})=B_E$. It follows that the adjoint map $S^\ast:E'_\beta \to \ell_{\infty}(\Gamma)$ is an embedding. Then the map $S^\ast\circ T$ is $k$-continuous and weak$^\ast$ continuous (because $S^\ast$ and $T$ are weak$^\ast$ continuous). The map $S^\ast\circ T$ is also an almost $k$-compact  {\rm(}resp., $k$-sequential{\rm)} map because if $K$ is a compact subset (resp., a null-sequence) in $\CC(X)$ such that $\{T_E(x,a):a\in B_E\} \subseteq \cacx(K)$, then
\[
\Big\{(S^\ast\circ T)_E(x,a):a\in B_{\ell_{1}(\Gamma)}\Big\}= \Big\{ \langle T(x),S(a)\rangle: a\in B_{\ell_{1}(\Gamma)}\Big\} = \{T_E(x,a):a\in B_E\} \subseteq \cacx(K).
\]
Therefore, by (ii), $S^\ast\circ T$ is continuous. Since $S^\ast$ is an embedding, we obtain that also $T$ is continuous.
\smallskip

(iii)$\Ra$(ii) is trivial.
\smallskip

(ii)$\Ra$(i) Suppose for a contradiction that $X$ is not an Ascoli space (resp., a sequentially Ascoli space). Then there is a compact set (resp., a null sequence) $K$ in $\CC(X)$ which is not equicontinuous at some point $x_0\in X$. Therefore there is $\e>0$ such that for every neighborhood $U$ of $x_0$ there exists a function $f_U\in K$ and a point $x_U\in U$ such that $|f_U(x_U)-f_U(x_0)|\geq \e$. Set $M:=2\sup\{|f(x_0)|+\e: f\in K\}$. Denote by $Q:\CC(X)\to\CC(X)$ the co-restriction map defined in Lemma \ref{l:restriction-Ck}. Then, by Lemma \ref{l:restriction-Ck},  $Q$ is continuous. Therefore $Q(K)$ is a compact subset of $\CC(X)$ such that
\begin{itemize}
\item[(a)] $\|R\big(Q(f)\big)\|_X \leq M$ for every $f\in K$;
\item[(b)] $R\big(Q(K)\big)$ is not equicontinuous at $x_0$ by the definition of $M$.
\end{itemize}
Hence, replacing $K$ by $Q(K)$, we can assume  that the following conditions hold:
\begin{itemize}
\item[(c)] there is $M>0$ such that $\|f\|_X \leq M$ for every $f\in K$.
\item[(d)] for every neighborhood $U$ of $x_0$ there exists a function $f_U\in K$ and a point $x_U\in U$ such that $|f_U(x_U)-f_U(x_0)|\geq \e$.
\end{itemize}

By (c) we can define a map $T:X\to \ell_\infty (K)$ by $T(x):=\big(f(x)\big)_{f\in K}$.
\smallskip

{\em Claim 1. The map $T$ is $k$-continuous.} Indeed, let $A$ be a compact subset of $X$, and let $x_0\in A$.
Denote by $R_A:\CC(X)\to C(A)$ the restriction map defined by $R_A(f):=f{\restriction}_A$. It is clear that $R_A$ is continuous. Therefore $R_A(K)$ is a compact subset of the Banach space $C(A)$. Since $A$ is an Ascoli space, it follows that $R_A(K)$ is equicontinuous. Hence for every $\e>0$, there is a neighborhood $U_\e(x_0)\subseteq X$ of $x_0$ such that
\begin{equation} \label{equ:Ascoli-char-1}
|f(x)-f(x_0)|<\e \;\; \mbox{ for every } \; x\in U_\e(x_0)\cap A \; \mbox{ and each } \; f\in K.
\end{equation}
Then (\ref{equ:Ascoli-char-1}) implies
\[
T(x)-T(x_0)=\big(f(x)-f(x_0)\big)_{f\in K} \in \e B_{\ell_\infty(K)} \;\; \mbox{ for every }\; x\in U_\e(x_0)\cap A,
\]
which means that $T{\restriction}_A$ is continuous.
\smallskip

{\em Claim 2. The map $T$ is weak$^\ast$ continuous.} Indeed, fix an arbitrary $x_0\in X$. Let $\mu=\sum_{n\in\w} a_n \delta_{f_n} \in \ell_1(K)$, where $f_n\in K$ and $f_n\not= f_m$ for distinct $n,m\in\w$,  and let $\e>0$. Choose $N\in\w$ such that
\begin{equation} \label{equ:Ascoli-char-2}
\sum_{n>N} |a_n| <\tfrac{\e}{4M}.
\end{equation}
Take a neighborhood $V(x_0)$ of $x_0$ such that
\begin{equation} \label{equ:Ascoli-char-3}
|f_n(x)-f_n(x_0)|<\tfrac{\e}{2\cdot\|\mu\|+2} \;\; \mbox{ for each } \; x\in V(x_0)\; \mbox{ and all } \; n\leq N.
\end{equation}
Then for every $x\in V(x_0)$, (\ref{equ:Ascoli-char-2}) and (\ref{equ:Ascoli-char-3}) and (c) imply
\[
\begin{aligned}
\big|\langle\mu,T(x)-T(x_0)\rangle & = \Big| \sum_{n\in\w} a_n\big(f_n(x)-f_n(x_0)\big)\Big|\\
& \leq \sum_{n\leq N} |a_n|\cdot \tfrac{\e}{2\cdot\|\mu\|+2} +\sum_{n>N} |a_n|\cdot 2M <\tfrac{\e}{2}+\tfrac{\e}{2}=\e,
\end{aligned}
\]
which means that $T$ is weak$^\ast$ continuous, as desired.
\smallskip

\smallskip

{\em Claim 3. The map $T$ is almost $k$-compact} (resp., {\em almost $k$-sequential}). Indeed, the definition of the Banach space $\ell_1(K)$ implies that  the set
\[
\{ T_E(x,a):\; a\in B_{\ell_1(K)}\}= \{ \langle T(x),a\rangle: \; a\in B_{\ell_1(K)}\}
\]
is the absolutely convex hull of the set
\[
\{ \langle T(x),\delta_f\rangle: \; f \in K\} = K
\]
which is a compact set (resp., a null sequence) in $\CC(X)$. The claim is proved.
\smallskip

Taking into account Claims 1--3, (ii) implies that the map $T:X\to \ell_\infty (K)$ is continuous. In particular, there  is a neighborhood $U$ of the point $x_0$ such that
\[
T(x)-T(x_0) \in \tfrac{\e}{2} B_{\ell_\infty (K)} \;\; \mbox{ for every }\; x\in U,
\]
which means that $|f(x)-f(x_0)|\leq \tfrac{\e}{2}$ for every $f\in K$. But this contradicts the condition (d).\qed
\end{proof}


\bibliographystyle{amsplain}

\end{document}